\documentclass[12pt]{amsart}
\usepackage{amscd,amsmath,amsthm,amssymb,graphics}
\usepackage{amscd,amsmath,amsthm,amssymb,graphics}
\usepackage{pstcol,pst-plot,pst-3d}
\usepackage{lmodern,pst-node}

\usepackage{multicol}
\usepackage{amssymb}
\usepackage{xcolor}
\usepackage{hyperref}
\usepackage{mathtools}

\usepackage{epic,eepic}
\usepackage{amsfonts,amssymb,amscd,amsmath,enumerate,verbatim}
\usepackage{pstcol,pst-plot,pst-3d}

\usepackage{enumerate}
\psset{unit=0.7cm,linewidth=0.8pt,arrowsize=2.5pt 4}

\newpsstyle{fatline}{linewidth=1.5pt}
\newpsstyle{fyp}{fillstyle=solid,fillcolor=verylight}
\definecolor{verylight}{gray}{0.97}
\definecolor{light}{gray}{0.9}
\definecolor{medium}{gray}{0.85}
\definecolor{dark}{gray}{0.6}

\def\frk{\frak}               

\def\mm{{\frk m}}

\def\Phi{{\frk n}}
\def\Phi{{\frk N}}
\def\MM{{\mathcal M}}

\def\opn#1#2{\def#1{\operatorname{#2}}} 
\opn\chara{char} \opn\length{\ell} \opn\pd{pd} \opn\rk{rk}
\opn\projdim{proj\,dim} \opn\injdim{inj\,dim} \opn\rank{rank}
\opn\depth{depth} \opn\grade{grade} \opn\height{height}
\opn\embdim{emb\,dim} \opn\codim{codim}

\opn\Tr{Tr} \opn\bigrank{big\,rank}
\opn\superheight{superheight}\opn\lcm{lcm}
\opn\trdeg{tr\,deg}
\opn\reg{reg} \opn\lreg{lreg} \opn\ini{in} \opn\lpd{lpd}
\opn\size{size}\opn\bigsize{bigsize}
\opn\cosize{cosize}\opn\bigcosize{bigcosize}
\opn\sdepth{sdepth}\opn\sreg{sreg}
\opn\link{link}\opn\fdepth{fdepth}
\opn\rank{rank}
\opn\Deg{Deg}
\opn\msupp{msupp}
\opn\div{div} \opn\Div{Div} \opn\cl{cl} \opn\Cl{Cl}
\let\epsilon\varepsilon
\let\phi=\varphi
\let\kappa=\varkappa
\opn\Spec{Spec} \opn\Supp{Supp} \opn\supp{supp} \opn\Sing{Sing}
\opn\Ass{Ass} \opn\Min{Min}\opn\Mon{Mon} \opn\dstab{dstab} \opn\astab{astab}
\opn\Syz{Syz}
\opn\Ann{Ann} \opn\Rad{Rad} \opn\Soc{Soc}
\opn\Im{Im} \opn\Ker{Ker} \opn\Coker{Coker} \opn\Am{Am}
\opn\Hom{Hom} \opn\Tor{Tor} \opn\Ext{Ext} \opn\End{End}
\opn\Aut{Aut} \opn\id{id}

\opn\nat{nat}
\opn\pff{pf}
\opn\Pf{Pf} \opn\GL{GL} \opn\SL{SL} \opn\mod{mod} \opn\ord{ord}
\opn\Gin{Gin} \opn\Hilb{Hilb}\opn\sort{sort}
\opn\initial{init}
\opn\ende{end}
\opn\height{height}
\opn\type{type}
\opn\set{set}
\opn\aff{aff} \opn\con{conv} \opn\relint{relint} \opn\st{st}
\opn\lk{lk} \opn\cn{cn} \opn\core{core} \opn\vol{vol}
\opn\link{link} \opn\star{star}\opn\lex{lex}
\opn\gr{gr}

\def\pot#1#2{#1[\kern-0.28ex[#2]\kern-0.28ex]}

\opn\dirlim{\underrightarrow{\lim}}
\opn\inivlim{\underleftarrow{\lim}}
\def\Implies{\ifmmode\Longrightarrow \else
        \unskip${}\Longrightarrow{}$\ignorespaces\fi}
\def\implies{\ifmmode\Rightarrow \else
        \unskip${}\Rightarrow{}$\ignorespaces\fi}
\def\iff{\ifmmode\Longleftrightarrow \else
        \unskip${}\Longleftrightarrow{}$\ignorespaces\fi}

\let\:=\colon
 \theoremstyle{plain}
\newtheorem{Theorem}{Theorem}[section]
\newtheorem*{Theorem*}{Theorem}
 \newtheorem{Lemma}[Theorem]{Lemma}
 \newtheorem{Corollary}[Theorem]{Corollary}
 \newtheorem{Proposition}[Theorem]{Proposition}

 \theoremstyle{definition}
 \newtheorem{Definition}[Theorem]{Definition}
 \newtheorem{Remark}[Theorem]{Remark}
 \newtheorem{Remarks}[Theorem]{Remarks}
 \newtheorem{Example}[Theorem]{Example}

\let\epsilon\varepsilon
\let\kappa=\varkappa
\opn\dis{dis}
\def\pnt{{\raise0.5mm\hbox{\large\bf.}}}

\opn\Lex{Lex}

\begin{document}
\title{On the matroidal path ideals}
\author {Mehrdad Nasernejad, Kazem Khashyarmanesh, Ayesha Asloob Qureshi}

\address{Kazem Khashyarmanesh, Mehrdad Nasernejad, Department of Pure Mathematics, Ferdowsi University of Mashhad, P.O.Box 1159-91775, Mashhad, Iran}\email{ m$\_$nasernejad@yahoo.com, khashyar@ipm.ir}

\address{Ayesha Asloob Qureshi, Sabanc\i \; University, Faculty of Engineering and Natural Sciences, Orta Mahalle, Tuzla 34956, Istanbul, Turkey}\email{aqureshi@sabanciuniv.edu}


\begin{abstract}
We prove that the set of all paths of a fixed length in a complete multipartite graph is the bases of a matroid. Moreover, we discuss the Cohen-Macaulayness and depth of powers of $t$-path ideals of a complete multipartite graph. 
\end{abstract}

\thanks{The third author was funded by the project 118F321 under the program 2509 of the Scientific and Technological Research Council of Turkey(TUBITAK).}


\maketitle

\section*{Introduction}


Matroid theory is one of the most attractive areas in combinatorics and is deeply rooted in graph theory and linear algebra. The well-known examples of matroids include uniform matroids, graphical matroids and linear matroids. We refer to \cite{W}, \cite{O} and \cite{HHdis} for basic definitions and notions in matroid theory. The polymatroids originated in \cite{E} and the {\em discrete polymatroids} appeared in \cite{HHdis} as a multiset analogue of matroids. In commutative algebra, the matroidal ideals as the squarefree version of polymatroidal ideals, hold a very special place due to their nice algebraic and homological properties. A monomial ideal $I \subset K[x_1, \ldots, x_n]$ is called {\em polymatroidal}, if the set of exponent vectors of the minimal generating set of $I$ corresponds to the set of bases of a discrete polymatroid. In particular, a squarefree polymatroidal ideal is simply referred to as a {\em matroidal} ideal, because in this case, the set of exponent vectors of the minimal generating set of $I$ corresponds to the set of bases of a matroid. The algebraic and homological properties of polymatroidal ideals have been studied by many authors, for example, see \cite{HY}, \cite{HRV}, \cite{HH}, \cite{BH}, \cite{HQ}. It is known from \cite{CH} and \cite{HY} that the product of polymatroidal ideals is again polymatroidal. Moreover, polymatroidal ideals have linear quotients, and hence linear resolutions, see \cite{CH}. In particular, all powers of a polymatroidal ideal have linear resolutions. If $I$ is a polymatroidal ideal, then the Rees algebra $\mathcal{R}(I)$ is normal, see \cite[Proposition 3.11]{V}, \cite[Theorem 3.4]{HRV}. Furthermore, polymatroidal ideals are known to have the persistence property, that is, $\Ass(I^k)\subset \Ass(I^{k+1})$ for all $k$, see \cite[Proposition 3.3]{HRV}, and they even have strong persistence property, that is, $I^{k+1}:I=I^k$ for all $k$, see \cite[Proposition 2.4]{HQ}.

In this paper, we introduce another class of matroidal ideals, namely, the class of the $t$-path ideals of complete multipartite graphs. Let $G$ be a finite simple graph with $n$ vertices.  We refer to a path of length $t-1$ in $G$ as a $t$-path. Let $S$ be a polynomial ring over a field $K$ in $n$ variables.  To simplify the notation, throughout this text, we identify the vertices of $G$ with the variables in $S$.  The {\em $t$-path ideal} $I_t(G) \subset S=K[x_1, \ldots, x_n]$ is generated by those monomials $x_{i_1}\cdots x_{i_t}$ such that $x_{i_1}, \ldots, x_{i_t}$ is a $t$-path in $G$. The $t$-path ideals of graphs were introduced by Conca and De Negri in \cite{CN} and later on discussed by several authors for different classes of graphs, in particular for directed trees, trees and cycles, see \cite{HT1}, \cite{BHK}, \cite{CGMPT}. When $t=2$, then $I_2(G)$ is the {\em edge ideal} of $G$ which is usually denoted as $I(G)$. Let $K_{n_1, \ldots, n_r}$ be the complete $r$-partite graph with $r$ partition sets of sizes $n_1, \ldots, n_r$. The main result of this paper is\\

\textbf{Theorem \ref{Th.3}.} \;
Let $t \geq 2$. If $I_{t}(K_{n_1,\ldots,n_r}) \neq 0$, then it  is a matroidal ideal. 
\newline

In other words, let $t \geq 2$ and $\MM_t(G)$ be a collection of subsets of $V(G)$ such that $A \in \MM_t(G)$ if and only if there exists a $t$-path in $G$ with vertex set $A$. Then it follows from Theorem~\ref{Th.3} that if $\MM_t(K_{n_1, \ldots, n_r})\neq \emptyset$, then it is the set of bases of a matroid. Up to our knowledge there is not much known about $t$-path ideals. In our case, after establishing Theorem~\ref{Th.3}, the $t$-path ideals of a complete $r$-partite graph inherit all the nice properties of matroidal ideals. 

The contents of this paper are distributed as following: in Section~\ref{prem}, the definitions and basic facts about $t$-paths and $t$-path ideals of complete $r$-partite graphs are provided. The Section~\ref{matroidpath} is devoted to establish the matroidal property of $I_t(K_{n_1, \ldots, n_r})$ in Theorem~\ref{Th.3}. It follows from  \cite[Theorem 3.2]{CW} and \cite[Theorem 2.2]{AY} that an edge ideal $I(G)$ is matroidal if and only if $G$ is a complete $r$-partite graph. In Theorem~\ref{Th.2}, we give another proof of this fact. 
It is natural to ask that if the $t$-path ideal of a graph is matroidal for some $t$, then does it follow that its non-zero $k$-path ideal is also matroidal for $k \geq t$? In Example~\ref{exp2}, we provide a graph whose 3-path ideal is matroidal but $4$-path ideal is not matroidal.

The identification of $I_t(K_{n_1, \ldots, n_r})$ as a matroidal ideal, helps tremendously to study its algebraic properties. In Section~\ref{sec:CM}, the Cohen-Macaulay property of $I_t(K_{n_1, \ldots, n_r})$ is discussed. It is known from \cite[Theorem 4.2]{HH} that a matroidal ideal is Cohen-Macaulay if and only if it is a principal ideal or a squarefree Veronese ideal. It remains to investigate for which $t$ and under which conditions on $n_1, \ldots, n_r$, the ideal $I_t(K_{n_1, \ldots, n_r})$ is squarefree Veronese. As a main result of Section~\ref{sec:CM}, we prove the following:\\

\textbf{Theorem \ref{CM}.} \; The $t$-path ideal $I_t(K_{n_1, \ldots, n_r}) \neq 0$ is Cohen-Macaulay if and only if $n_i \leq \lceil t/2\rceil$, for each $i=1, \ldots, r$. 
\newline

The stability indices related to associated primes and depth of powers of polymatroidal ideals are of particular interest and have been discussed in several papers, for example, see \cite{HRV,HQ,HV}. 
The limit depth and index of depth stability is computed for $t$-path ideals of complete bipartite graphs in Theorem~\ref{bipartite}, and for 3-path ideals of complete $r$-partite graphs in Theorem~\ref{v>5}. All of these concepts are defined and discussed in Section~\ref{limdepth}. From Theorem~\ref{bipartite} and Theorem~\ref{v>5}, it is evident that there is no uniform description for the depth of the powers of $I_t(K_{n_1, \ldots, n_r})$. However, in Theorem~\ref{t=4},  it is observed that if $r \geq 3$ and all of the $n_i$'s are bigger than or equal to $\lceil t/2 \rceil$, then the limit depth of $I_t(K_{n_1, \ldots, n_r})$ is 0.

\section{Preliminaries}\label{prem}

First we recall some basis definition and notion related to matroids. Let $[n]=\{1,\ldots,n\}$ and $P([n])$ denote the set of all subsets of $[n]$. For any $S \subseteq [n]$, the cardinality of $S$ is denoted by $|S|$. A {\em matroid} $\mathcal{M}$ on the ground set $[n]$ is a non-empty collection of subsets of $[n]$ satisfying the following properties:

\begin{enumerate}
\item if $A \in \mathcal{M}$ and $B \subset A$, then $B \in \mathcal{M}$;
\item if $A,B \in \mathcal{M}$ and $|B| \leq |A|$, then there exists $a\in A \setminus B$ such that $B\cup\{a\} \in \mathcal{M}$.
\end{enumerate}

A maximal element (with respect to inclusion) in $\mathcal{M}$ is called a {\em base}. Condition (2) can be used to show that all bases of $\mathcal{M}$ have the same cardinalities. Let $\mathcal{B}$ be the set containing all the bases of $\mathcal{M}$, then $\mathcal{B}$ is distinguished by the following ``exchange property".
\begin{enumerate}
\item[{(EP)}]  For any $A,B \in \mathcal{B}$, if $a \in A\setminus B$, then there exists $b\in B \setminus A$ such that\\ $(A\setminus\{a\} )\cup\{b\} \in \mathcal{B}$.
\end{enumerate}

Given any subset $\mathcal{C}$ of $P([n])$, there exists a matroid on $[n]$ with $\mathcal{C}$ as its sets of bases if and only if $\mathcal{C}$ satisfies the exchange property. 


Let $K$ be a field and $S=K[x_1,\ldots, x_n]$ be a polynomial ring over $K$. Let $I$ be a monomial ideal in $S$. The set $\mathcal{G}(I)$ denotes the unique minimal set of monomial generators of $I$. For a given ${\bf a}=(a_1,\ldots, a_n)$ with non-negative entries, ${\bf x}^{\bf a}$ denotes the monomial $x_1^{a_1}\cdots x_n^{a_n}$ in $S$. 
A monomial ideal $I \subset S$ generated in a single degree is called {\em polymatroidal} if for all monomials ${\bf x}^{\bf a}, {\bf x}^{\bf b}\in \mathcal{G}(I)$ with $a_i>b_i$, there exists $j$ with $a_j<b_j$ such that $x_j({\bf x}^{\bf a}/x_i )\in \mathcal{G}(I)$. A squarefree polymatroidal ideal is simply referred to as a {\em matroidal} ideal. In other words, a monomial ideal $I$ in $S$ is matroidal, if $\mathcal{G}(I)$ can be identified as a set of bases of a matroid.  




Next we recall some basic definitions and notions from graph theory. All graphs considered in this paper will be simple, undirected and finite. Let $G$ be a graph with the vertex set $V(G)$ and the edge set $E(G)$.  A {\em path} in $G$ is a sequence of distinct vertices $x_{i_1},\ldots, x_{i_{t}}$ such that $\{x_{i_j}, x_{i_{j+1}}\}\in E(G)$ for $j=1,\ldots,t-1$.  The length of a path is the number of edges in it. We will refer to a path of length $t-1$ as a $t$-path. In other words, a $t$-path in $G$ is a path with $t$ vertices. A subset $A \subseteq V(G)$ is called {\em independent} if no vertices in $A$ are adjacent in $G$. 

A graph $G$ is called {\em complete r-partite}, if $V(G)$ can be partitioned into $r$ independent sets $V_1, \ldots, V_r$ such that $a$ and $b$ are adjacent for all $a\in V_i$ and all $b\in V_j$ with $1 \leq i \neq j \leq r$. Such a partition is called an {\em $r$-partition} of $G$. Moreover, if $|V_i|=n_i$ for all $i=1, \ldots, r$, then the complete $r$-partite graph is denoted by $K_{n_1,\ldots,n_r}$.



Let $G$ be a graph with $|V(G)|=n$. As mentioned in the introduction, we identify the vertices of the graph $G$ with the variables in the polynomial ring $S = K[x_1, \ldots, x_n]$.  
 \begin{Definition}
The {\it $t$-path ideal} of $G$, denoted by $I_t(G)$, is 
\[
I_t(G) : = (x_{i_1}\cdots x_{i_{t}} \ : x_{i_1},\ldots, x_{i_{t}} \text{ is a $t$-path in $G$}) \subset S.
\]
 \end{Definition}
 When $t=2$, then $I_2(G)$ is simply the {\em edge ideal}  $I(G)$ of $G$. If there is no $t$-path in $G$, then we set $I_t(G)=0$. 
 
 The following example illustrates the definitions that have been stated above.
\begin{Example}\label{exp1}
Let $t=4$ and $K_{1,2,3}$ be as shown in Figure~\ref{exp1} with partition sets $\{x_1\}$, $\{x_2,x_3\}$ and $\{x_4,x_5,x_6\}$. 
\begin{figure}[htbp]\label{exp1}
\includegraphics[width = 5cm]{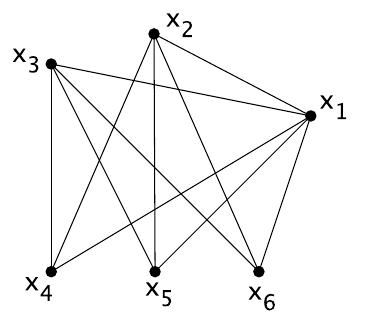}
\caption{ The graph $K_{1,2,3}$}
\end{figure}
Then $I_4(K_{1,2,3})$ has the following generators:
\[
x_1x_2x_3x_4,x_1x_2x_3x_5,x_1x_2x_3x_6, x_1x_2x_4x_5,x_1x_2x_4x_6,x_1x_2x_5x_6, \]\[x_1x_3x_4x_5, x_1x_3x_4x_6, x_1x_3x_5x_6, x_2x_3x_4x_5,x_2x_3x_4x_6,x_2x_3x_5x_6.
\]
\end{Example}

Recall  that if  $u=x_{1}^{a_{1}}\cdots x_{n}^{a_{n}}$ is  a
monomial in $S=K[x_{1},\ldots ,x_{n}]$, then the {\em support} of $u$ is given by $\supp(u)=\{x_i : ~a_{i}>0\} $. We set $\deg_{x_i}u=a_i$, for all $i=1, \ldots, n$.
For a monomial ideal $I\subset S$, the {\em support} of $I$ is $\supp(I)=\bigcup_{u \in \mathcal{G}(I)}\supp(u)$. If $\supp(I)=\{x_1, \ldots, x_n\}$, then $I$ is said to be {\em fully supported}. If $\mathcal{G}(I)=\{u_1, \ldots, u_m\}$, then $\gcd(I)=\gcd(u_1, \ldots, u_m)$.

Below we give a list of remarks that will be used frequently throughout the paper. For any integer $a$, the notations $\lceil a  \rceil$ and $\lfloor a \rfloor$ denote the ceiling and floor functions of $a$, respectively. Let $G=K_{n_1, \ldots, n_r}$ and $V_1, \ldots, V_r$ be the $r$-partition of $G$. Then we have the following. 

\begin{Remarks} \label{Rem.1}
Let $P:x_{i_1},x_{i_2},\ldots,x_{i_{t}}$  be a $t$-path in $G$, and $u=x_{i_1}x_{i_2}\cdots x_{i_{t}} \in \mathcal{G}(I_t(G))$. If $x_{i_j} \in V_s$ for some $1 \leq s \leq r $ and $2\leq j \leq t-1$, then $x_{i_{j-1}}, x_{i_{j+1}}\notin V_s$. This leads to the following conclusions:
\begin{enumerate}
\item[{(i)}]  A $t$-path in $G$ can have at most $\lceil t/2\rceil$ vertices in $V_s$.
\item[{(ii)}]  Let $x_{i_j} \in V_s$ for some $1\leq j \leq t$. If there exists $x_k \in V_s$ such that $x_k \notin V(P)$, then $x_{i_j}$ can be replaced by $x_k$ in $P$ to obtain a new $t$-path $P'$ in $G$. In particular, $(u/{x_{i_j}})x_k \in \mathcal{G}(I_t(G))$.
\item[{(iii)}] If there exists some $V_s$ such that $V(P) \cap V_s = \emptyset$, then any $x_{i_j} \in V(P)$ can be replaced in $P$ by any $x_k \in V_s$, to obtain a new path $P'$ in $G$. In particular, $(u/x_{i_j})x_k \in \mathcal{G}(I_t(G))$, for all $x_k \in V_s$ and for each $j=1, \ldots, t$. 
\item[{(iv)}] If $I_t(G)\neq 0$, then $I_t(G)$ is fully supported. This can be easily seen due to statements (ii) and (iii). 
\item[{(v)}] Let $I_t(G)\neq 0$. Then  $I_t(G)$ is a principal ideal if and only if $t=n_1+\cdots+n_r$. This also follows from statements (ii) and (iii). 
\end{enumerate}
\end{Remarks}

\section{The Matroidal path ideals of graphs}\label{matroidpath}

It is proved in \cite[Theorem 1.1]{BH} that a monomial ideal $I \subset S=K[x_1,\ldots,x_n]$ is polymatroidal if and only if $I:u$ is polymatroidal for all monomials $u$ in $S$. In particular, $I:x_i$ is polymatroidal if $I$ is polymatroidal. The following theorem gives further information about $I:x_i$.  
 \begin{Theorem} \label{Th.1}
 Let $I\subset S=K[x_1,\ldots,x_n]$ be a   polymatroidal ideal    and $1\leq i \leq n$. Let $I=\sum_{j=0}^d I_jx^{j}_i$, where $d=\max\{\deg_{x_i}u: u \in \mathcal{G}(I)\}$ and
 $I_j=(u/x^{j}_i: u\in \mathcal{G}(I), ~ \mathrm{deg}_{x_i}u=j)$ for each $j=0, \ldots, d$. Then $I_0 \subseteq I_1 \subseteq \cdots \subseteq I_d$. 
  \end{Theorem}
 \begin{proof}
 If $d=0$, then $I=I_0$ and there is nothing to prove. Let $d \geq 1$ and set 
 $J_1:= \sum_{j=1}^d I_jx^{j-1}_i$. It is easy to see that $I=I_0+x_iJ_1$ and $(I:x_i)=I_0+J_1$. 
 Following the proof of \cite[Theorem 1.1]{BH}, one can deduce that  $I_0 \subseteq J_1$, and hence $(I:x_i)=J_1$ is a polymatroidal ideal. Now, on account of $I_0= (u: u\in \mathcal{G}(I), ~ \mathrm{deg}_{x_i}u=0)$ and $I_1=(u/x_i: u\in \mathcal{G}(I), ~ \mathrm{deg}_{x_i}u=1)$, we obtain $I_0 \subseteq I_1$. If $d=1$, then we are done. Otherwise, if $d\geq 2$, we set $J_2:= \sum_{j=2}^d I_jx^{j-2}_i$. Thus, $J_1=I_1+x_iJ_2$. Once again, from the proof of \cite[Theorem 1.1]{BH} we obtain that $I_1 \subseteq J_2$, and thus  $I_1 \subseteq I_2$.  Therefore,   $I_0\subseteq I_1 \subseteq I_2$.
 Now, set  $J_k=\sum_{j=k}^d I_jx^{j-k}_i$, where $1\leq k \leq d$. By continuing this procedure, we can conclude that 
  $I_0 \subseteq I_1 \subseteq \cdots \subseteq I_d$, as required.  
 \end{proof}
If $I \subset K[x_1,\ldots,x_n]$ is a matroidal ideal, then $\max\{\deg_{x_i}u: u \in \mathcal{G}(I)\}$ is either 0 or 1. For each variable $x_i$, we set $I_{0,i}=(u: x_i \notin \supp(u))$ and $I_{1,i}=(u/x_i: x_i \in \supp(u))$. Then, we have $I=I_{0,i}+x_iI_{1,i}$, for all $i=1, \ldots, n$. With this notation, we have the following specialization of Theorem~\ref{Th.1}.

\begin{Proposition} \label{Pro.1}
 Let $I\subset K[x_1,\ldots,x_n]$ be a fully supported matroidal ideal generated in degree $d$ with $\gcd(I)=1$. We have the following:
 \begin{enumerate}
 \item[{\em(i)}] for each $i=1, \ldots, n$ we have $\mathcal{G} (I_{1, i}) \subseteq \cup_{t=1, ~t \neq i}^n \mathcal{G} (I_{1, t})$. 
\item[{\em(ii)}]  if $I_{1,i}\subseteq  I_{1,j}$ for some $1\leq i, j \leq n$ with $i\neq j$. Then $I_{1,i}= I_{1,j}$.
\end{enumerate}
  \end{Proposition}
  
\begin{proof}
(i) The assumption that  $I$ is fully supported gives that $I_{1,i} \neq 0$, for each $i=1, \ldots, n$. Moreover, $\gcd(I)=1$ gives that $I_{0,i} \neq 0$, for each $i=1, \ldots, n$. Let $u'\in \mathcal{G}(I_{1,i})$ and $v\in \mathcal{G}(I_{0,i})$. Then it follows from the definition of $I_{0,i}$ and $I_{1,i}$ that $v\in \mathcal{G}(I)$ and $u:=u' x_i \in \mathcal{G}(I)$. Since $\deg_{x_i}u > \deg_{x_i}v$, one can conclude from the exchange property that there exists some $s$ with $\deg_{x_s}u < \deg_{x_s}v$ such that $u'x_s =x_s(u/x_i) \in \mathcal{G}(I)$. Then, $u'\in \mathcal{G}(I_{1,s})$. We therefore have $\mathcal{G} (I_{1, i} )\subseteq \cup_{t=1, ~t \neq i}^n \mathcal{G} (I_{1, t})$, as claimed. 

(ii)  Suppose, on the contrary, that there exists some monomial 
 $u \in \mathcal{G}(I_{1,j}) \setminus \mathcal{G}(I_{1,i})$. Thus, $x_ju  \in \mathcal{G}(I)$. If $x_i\in \supp(u)$, then $(x_ju)/x_i \in  \mathcal{G}(I_{1,i})$. Using the assumption $I_{1,i}\subseteq  I_{1,j}$, we obtain $(x_ju)/x_i \in I_{1,j}$ which gives $(x_j^2u)/x_i \in I$. Since $I$ is matroidal, $(x_ju)/x_i \in I$. This yields a contradiction to $x_ju \in \mathcal{G}(I)$. Hence, $x_i\notin \supp(u)$. Since $x_ju \in \mathcal{G}(I)$, we conclude that $x_ju \in \mathcal{G}(I_{0,i})$. From Theorem~\ref{Th.1},  we know that $I_{0,i} \subset I_{1,i}$. Then there exists some monomial $w \in \mathcal{G}(I_{1,i})$ such that $w$ divides $x_ju$. Again, by using the assumption, $I_{1,i}\subseteq  I_{1,j}$, we have $x_j \notin \supp(w)$ and therefore, $w| u$. Since both $w$ and $u$ are monomials of degree $d-1$, we conclude that $u=w$, as required.  \end{proof}

It follows from \cite[Theorem 3.2]{CW}  that if $I(G)$ is a matroidal ideal, then $G$ is a complete multipartite graph. The converse follows as a corollary of \cite[Theorem 2.2]{AY}. We give another straightforward proof of this fact in the language of edge ideals of complete multipartite graphs.

 \begin{Theorem} \label{Th.2}
Let $G$ be a graph without isolated vertices. The edge ideal $I(G)$ is a matroidal ideal if and only if $G$ is a complete $r$-partite graph for some $r\geq2$.
 \end{Theorem}

 \begin{proof}
We first assume that $G$ is a complete $r$-partite graph, that is, $G=K_{n_1, .., n_r}$ with vertex partition $V_1, \ldots, V_r$ and $|V_i|=n_i$ for all $i=1,\ldots, r$. Let $u:=x_k x_l \in I(G)$  and   $v:=x_ix_j \in I(G)$ such that $x_k$ does not divide $v$. We need to show that either $x_i(u/x_k)=x_ix_l \in I(G)$ or $x_j(u/x_k)=x_jx_l \in I(G)$. Let $x_l \in V_t$ for some $1 \leq t \leq r$.  The graph $G$ is complete $r$-partite and  $\{x_i,x_j\} \in E(G)$, hence it follows that at least one of the vertices between $x_i$ and $x_j$ does not belong to $V_t$. If $x_i \notin V_t$, then $x_ix_l \in I(G)$ and similarly if $x_j \notin V_t$, then $x_jx_l \in I(G)$, as required.

  Conversely, let  $I(G)$ be a matroidal ideal and $|V(G)|=n$. We can assume that $G$ does not contain any isolated vertices, and hence $I(G)$ is fully supported. If $I(G)$ is a principal ideal, then there is nothing to prove and the assertion holds trivially. Let $I(G)$ be not a principal ideal. If $\gcd(I(G)) = x_i$, for some variable $x_i$, then $G$ is a complete bipartite graph with $r$-partition $V_1=\{x_i\}$ and $V_2=V(G)\setminus\{x_i\}$, and again the assertion holds in this case.
  
  Finally, we assume that $\gcd(I)=1$. Note that for each $i=1, \ldots, n$, we have $I(G)_{1, i} =(x_q: x_q\in N_{G}(x_i))$, where $N_{G}(x_i)$ denotes the neighbourhood of $x_i$ in $G$. Then from Proposition~\ref{Pro.1}(ii), we conclude that if    $I(G)_{1, i} \subseteq    I(G)_{1, j}$, for some $i \neq j$ then  $I(G)_{1, i}=I(G)_{1, j}$; equivalently, if $ N_{G}(x_i) \subseteq  N_{G}(x_j)$, then $ N_{G}(x_i)= N_{G}(x_j)$. We can partition vertices of $G$ into disjoint sets, say $V_1, \ldots, V_r$ such that for all $k=1, \ldots, r$, the vertices in $V_k$ have same neighbourhood. Since $G$ is a simple graph, it is clear that each of the $V_k$ is an independent set of $G$. 
  We claim that $G=K_{n_1, \ldots, n_r}$ with $r$-partition $V_1, \ldots, V_r$. To prove our claim, it only remains to verify the following: if $x_i \in V_k$ and $x_j \in V_{\ell}$ with $k \neq \ell$, then   $\{x_i,x_j\} \in E(G)$. By virtue of  Proposition~\ref{Pro.1}(ii), we have $ N_{G}(x_i) \not\subseteq N_{G}(x_j)$ and $ N_{G}(x_j) \not\subseteq  N_{G}(x_i)$. Let $x_k \in N_{G}(x_i)\setminus N_{G}(x_j)$ and $x_\ell \in N_{G}(x_j)\setminus N_{G}(x_i)$. Then $\{x_i,x_k\}$ and $\{x_j,x_\ell\}$ are disjoint edges in $G$. Since $I(G)$ is matroidal, then by applying exchange property on $x_ix_k$ and $x_jx_\ell$, we conclude that $\{x_i,x_j\} \in E(G)$, as claimed. This completes the proof.  
\end{proof}


 To prove the main result of this section, we first need the following:
 \begin{Lemma}\label{Lem.2}
Let $G=K_{n_1,\ldots,n_r}$ and $t \geq 3$. Furthermore, let $P_1: x_{i_1},\ldots, x_{i_{t-1}}$ be a $(t-1)$-path in $G$ and $P_2:x_{j_1}, \ldots, x_{j_t}$ be a $t$-path in $G$. Then there exists  $x_{j_k} \in V(P_2)\setminus V(P_1)$ such that one of the following statements  is satisfied.
\begin{enumerate}
\item[\em{(1)}] $x_{j_k}, x_{i_1},\ldots, x_{i_{t-1}}$ is a $t$-path in $G$.
\item[\em{(2)}]  $x_{i_1},\ldots, x_{i_{t-1}},x_{j_k}$ is a $t$-path in $G$.
\item[\em{(3)}]  $t \geq 4$ and there exists $2\leq p \leq t-2$ such that $x_{i_1},\ldots x_{i_{p}}, x_{j_k}, x_{i_{p+1}}, \ldots, x_{i_{t-1}}$ is a $t$-path in $G$.
\end{enumerate}
\end{Lemma}

\begin{proof}
 Let $V_1 , \ldots, V_r$ be the $r$-partition of $G=K_{n_1,\ldots,n_r}$. We may assume that $x_{i_1} \in V_1$. If there exists $x_{j_k} \in V(P_2)\setminus V(P_1)$ such that $x_{j_k} \notin V_1$, then $\{x_{j_k}, x_{i_1}\} \in E(G)$, and $P:x_{j_k}, x_{i_1},\ldots, x_{i_{t-1}}$ is a $t$-path in $G$, as given in (1).

If (1) is not true, then every $x_{j_k} \in V(P_2)\setminus V(P_1)$ is such that $x_{j_k}  \in V_1$.   If  $x_{i_{t-1}} \notin V_1$, then for any  $x_{j_k} \in V(P_2)\setminus V(P_1)$, we have $\{x_{i_{t-1}},x_{j_k}\} \in E(G)$, and  hence we obtain a $t$-path $P:x_{i_1},\ldots, x_{i_{t-1}},x_{j_k} $ in $G$, as given in (2). Note that if $t=3$, then in $P_1:x_{i_1},x_{i_2}$, both vertices cannot be in $G_1$ because $G$ is a complete $r$-partite graphs. Hence either (1) or (2) must be true for $t=3$.

If both (1) and (2) do not hold, then $t \geq 4$. In this case, we try to construct a path as given in (3). The negation of (1) and (2) gives us that $x_{i_1}, x_{i_{t-1}} \in V_1$ and every $x_{j_k} \in V(P_2)\setminus V(P_1)$ is such that $x_{j_k}  \in V_1$. If there exists some $2\leq p\leq t-2$ such that $x_{i_p} ,x_{i_{p+1}} \notin V_1$, then for any $x_{j_k} \in V(P_2)\setminus V(P_1)$, we have $\{x_{i_p}, x_{j_k}\},  \{x_{j_k} , x_{i_{p+1}}\} \in E(G)$, and $P:x_{i_2},\ldots, x_{i_p}, x_{j_k},  x_{i_{p+1}}, \ldots, x_{i_t}$ is a $t$-path in $G$ as given in (3). 
Finally, suppose that each of (1), (2), and (3) do not hold. Then we are in the following situation: 
\begin{enumerate}
\item[(i)] $x_{i_1}, x_{i_{t-1}} \in V_1$;
\item[(ii)] every $x_{j_k} \in V(P_2)\setminus V(P_1)$ is such that $x_{j_k}  \in V_1$. This also gives that for every $x_{j_k} \in V(P_2)\setminus V_1$, we have $x_{j_k}\in V(P_1)$;
\item[(iii)]  for every $2 \leq p\leq t-2$ either $x_{i_p} \in V_1$ or $x_{i_{p+1}} \in V_1$.
\end{enumerate}

Since $x_{i_1} \in V_1$ and $\{x_{i_1}, x_{i_2}\} \in E(G)$, this implies  that $x_{i_2} \notin V_1$. It follows now from (iii) that $x_{i_3} \in V_1$. By continuing this process, we obtain from (iii) that $x_{i_{t-2}}\notin V_1$ since $x_{i_{t-1}} \in V_1$ and $\{x_{i_{t-2}}, x_{i_{t-1}}\}\in E(G)$. This shows that $P_1$ starts and ends at vertices in $V_1$, and all vertices in $P_1$ with odd indices are also in $V_1$. Furthermore, all vertices in $V(P_1)$ with even indices  do not belong to $V_1$. Hence, $t-1$ must be an odd integer, and in $V(P_1)$, there are $t/2-1$ vertices that do not belong to $V_1$, that is,  $|V(P_1)\setminus V_1| = t/2-1$. Moreover, it follows from (ii) that  every vertex in $V(P_2)$ which is outside of $V_1$ must belong to $V(P_1)$, that is,  $V(P_2)\setminus V_1 \subseteq V(P_1)\setminus V_1$. Since  $P_2$ is a $t$-path with even number of vertices in a complete $r$-partite graph $G$, this gives that  at least half of its vertices are not in $V_1$, that is,  $|V(P_2)\setminus V_1| \geq t/2$. This contradicts $|V(P_2)\setminus V_1|  \leq |V(P_1)\setminus V_1| = t/2-1$. This yields that  at least one of the statements in (1), (2),  or (3) must hold. 
 \end{proof}


Now we state the main result of this section.

\begin{Theorem} \label{Th.3}
Let $t \geq 2$. If $I_{t}(K_{n_1,\ldots,n_r}) \neq 0$, then it  is a matroidal ideal. 
\end{Theorem}

\begin{proof}
We set $G:=K_{n_1,\ldots,n_r}$ and let $V_1, \ldots, V_r$ be the $r$-partition of $G$. The assertion is true for $I_2(G)$, as shown in Theorem \ref{Th.2}.  Let $t \geq 3$ and $I_t(G) \neq 0$. If $I_t(G)$ is a principal ideal, then there is nothing to prove and the assertion holds trivially. Let $|\mathcal{G} (I_t(G))| \geq 2$. To prove that $I_{t}(G)$ is matroidal, we need to show that if $P_1: x_{i_1},\ldots, x_{i_t}$ and $P_2:x_{j_1}, \ldots, x_{j_t}$ are two $t$-paths with $x_{i_l} \notin V(P_2)$ for some $1\leq l \leq t$, then there exists some $x_{j_k} \in V(P_2)\setminus V(P_1)$, such that one obtains a new path by removing $x_{i_l}$ from $P_1$ and inserting $x_{j_k}$ in an appropriate position in $P_1$.  

We may assume that $x_{i_l} \in V_1$. If $x_{i_1},\ldots,x_{i_{l-1}}, x_{i_{l+1}}\ldots, x_{i_t}$ is $(t-1)$-path  in $G$ then we obtain the desired  conclusion by using Lemma~\ref{Lem.2}. In particular, if $l=1$ or $l=t$, then removing $x_{i_l}$ from $P_1$ gives a $(t-1)$-path in $G$. Now assume that $x_{i_1},\ldots,x_{i_{l-1}}, x_{i_{l+1}}\ldots, x_{i_t}$ is not a $(t-1)$-path  in $G$. This is possible only if $\{x_{i_{l-1}}, x_{i_{l+1}}\} \notin E(G)$, that is, $x_{i_{l-1}}, x_{i_{l+1}}$ belong to $V_a$ for some $2 \leq a \leq r$. If there exists some $x_{j_k} \in V(P_2)\setminus V(P_1)$ such that $x_{j_k} \notin V_a$, then  \[x_{i_1},\ldots, x_{i_{l-1}}, x_{j_k}, x_{i_{l+1}}, \ldots, x_{i_t}\] is a $t$-path in $G$, and the proof is complete. Otherwise, every $x_{j_k} \in V(P_2)\setminus V(P_1)$ is such that $x_{j_k} \in V_a$.  If $x_{i_1} \notin V_a$, then $\{x_{i_1},x_{i_{l+1}}\} \in E(G)$, and we get $x_{i_{l-1}}, x_{i_{l-2}}, \ldots, x_{i_1},x_{i_{l+1}}, \ldots, x_{i_{t}}$ as a $(t-1)$-path in $G$. We again use Lemma~\ref{Lem.2} to obtain the desired conclusion. A similar argument gives us the desired conclusion when $x_{i_t} \notin V_a$. Otherwise, we are in the following situation:

  


\begin{enumerate}
\item[{(1)}]  $x_{i_{l-1}} , x_{i_{l+1}}, x_{i_1}, x_{i_t} \in V_a$;
\item[{(2)}]   every $x_{j_k} \in V(P_2)\setminus V(P_1)$ is such that $x_{j_k} \in V_a$.  This  implies  that if $x_{j_k} \in V(P_2) \setminus V_a$, then $x_{j_k} \in V(P_1)$.
\end{enumerate} 
We  claim that there must exist  some $p$ with either $2\leq p \leq l-3$ or $ l+2 \leq p \leq t-2$ such that $x_{i_p}, x_{i_{p+1}} \notin V_a$. Indeed,  if our claim is true and   we can find such $p$ with  $2\leq p \leq l-3$, then 
$$x_{i_1},\ldots, x_{i_p}, x_{i_{l-1}}, x_{i_{l-2}},\ldots, x_{i_{p+1}},x_{i_{l+1}}, \ldots, x_{i_{t}}
$$
is a $(t-1)$-path in $G$, and the desired result can be deduced from  Lemma~\ref{Lem.2}. Similarly, if we can find such $p$ with $ l+2 \leq p \leq t-2$, then 

\[
x_{i_1},\ldots x_{i_{l-1}},x_{i_p}, x_{i_{p-1}}, \ldots,    x_{i_{l+1}},x_{i_{p+1}},\ldots, x_{i_{t}}
\]
is a $(t-1)$-path in $G$, and  one can conclude the desired result  from Lemma~\ref{Lem.2}. 

To establish  our claim, suppose, on the contrary, that 
for every  $p$ with $2\leq p \leq l-3$ and $ l+2 \leq p \leq t-2$, we have  either $x_{i_p}\in V_a$ or $x_{i_{p+1}} \in V_a$.
Since $x_{i_1} \in V_a$ and $\{x_{i_1},x_{i_2}\} \in E(G)$, this yields that $ x_{i_2} \notin V_a$. 
It follows now from  the assumption that  $x_{i_3} \in V_a$. 
Moreover, from (1) it follows  that 
$x_{i_{l-2}}\notin V_a$ since $x_{i_{l-1}} \in V_a$ and $\{x_{i_{l-2}},x_{i_{l-1}}\} \in E(G)$.  
   This forces $l-1$ to be an odd integer. Similarly, since $x_{i_{l+1}},  x_{i_t} \in V_a$, we derive that the number of vertices in the path $x_{i_{l+1}}, \ldots, x_{i_{t}}$ is odd as well. Collectively, we obtain that $t$ itself must be odd.  This shows that $V(P_1)$ contains $(t+1)/2$ vertices from $V_a$ and the other $(t-1)/2$  vertices are not from $V_a$. Note that these $(t-1)/2$  vertices that are not from $V_a$ include $x_{i_l}$ as well. Considering that, by (2),  every vertex in $V(P_2)\setminus V(P_1)$ belongs to $V_a$  and $x_{i_l} \notin V(P_2)$,  this implies  that in $V(P_2)$  there are at most $(t-1)/2-1$ vertices that are not from $V_a$, which is impossible by Remark~\ref{Rem.1}(i). Hence the claim holds. 
\end{proof}


It is known from \cite[Proposition 3.11]{HRV} that polymatroidal ideals are normal. In \cite[Corollary 2.11]{RV}, it is shown that the $t$-path ideals of complete bipartite graphs are normal. This result can also be seen now as a corollary of the above theorem. 

In general, for a given $t \geq 3$, one can find graphs that are not complete $r$-partite but their $t$-path ideal $I_t(G)$ is matroidal. As a very simple example, let $G$ itself be a path on $t$ vertices. Then  $I_t(G)$ is matroidal because it is a principal ideal, while $G$ is not complete $r$-partite.

Theorems~\ref{Th.2} and \ref{Th.3} together show that if $I_2(G)$ is matroidal then $I_t(G)$ is also matroidal for all $t \geq 2$. Therefore, it is natural to ask the following question: If $I_t(G)$ is matroidal for some $t$, then is it true that $I_{k}(G)$ is also matroidal for all $k \geq t$? The following example shows that it is not true in general. 

\begin{Example}\label{exp2}
Let $G$ be the graph as shown in the following figure. 
\begin{figure}[htbp]
\includegraphics[width = 3cm]{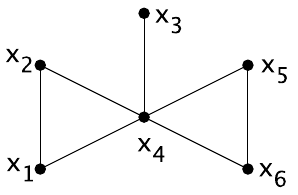}
\caption{The graph $G$}
\end{figure}
Then 
\[I_3(G)=(x_1x_2x_4,x_1x_3x_4,x_1x_4x_5,x_1x_4x_6,x_2x_3x_4,x_2x_4x_5
,x_2x_4x_6,\]
\[x_3x_4x_5,x_3x_4x_6,x_4x_5x_6).
\]
Note that $I_3(G)=x_4 J$, where
 \[
J=(x_1x_2,x_1x_3,x_1x_5,x_1x_6,x_2x_3,x_2x_5
,x_2x_6,x_3x_5,x_3x_6,x_5x_6).
\]
The ideal $J$ is indeed the edge ideal of the complete graph on the vertex set $\{x_1,x_2,x_3,x_5,x_6\}$, and by Theorem~\ref{Th.2}, $J$ is matroidal. This shows that $I_3(G)$ is also matroidal because it is the product of two matroidal ideals, namely $(x_4)$ and $J$.   Set $u:=x_1x_2x_3x_4$ and $v:=x_3x_4x_5x_6$. Then $u,v \in I_4(G)$, but  both $x_5(u/x_1)$ and $x_6(u/x_1)$ do not belong to $G$. Hence, $I_4(G)$ is not matroidal. 
\end{Example}
 
In the view of Theorem~\ref{Th.3} and Example~\ref{exp2}, we close this section with the following questions.   

\begin{enumerate}
\item For a given $t>2$, characterize the graphs for which $I_t(G)$ is matroidal.
\item For which graphs it is true that if $I_t(G)$ is matroidal for some $t$, then $I_{k}(G)$ is also matroidal for all $k \geq t$?
\end{enumerate}

\section{Cohen--Macaulay property of $I_t(K_{n_1, \ldots, n_r})$}\label{sec:CM}
Now we investigate the Cohen--Macaulay property for $I_t(K_{n_1, \ldots, n_r})$. Let $I$ be a monomial ideal in the polynomial ring $S=K[x_1, \ldots, x_n]$. The ideal $I$ is called the {\em Veronese} ideal of degree $d$ in $S$ if $I$ is generated by all monomial of degree $d$ in $S$. Similarly, one defines the squarefree counterpart of a Veronese ideal. The monomial ideal $I$ is called {\em squarefree Veronese} ideal of degree $d$ in $S$ if $I$ is generated by all squarefree monomials of degree $d$ in $S$.  It follows from \cite[Theorem 4.2]{HH} that a polymatroidal ideal $I$ is Cohen-Macaulay if and only if
$I$ is
\begin{itemize}
\item[(i)]  a principal ideal, or
\item[(ii)] a Veronese ideal, or
\item[(iii)] a squarefree Veronese ideal.
\end{itemize}

In particular, if $I$ is a matroidal ideal, then it is Cohen-Macaulay if and only if it is a principal ideal or squarefree Veronese ideal. From Remark~\ref{Rem.1} (v) it follows that $I_t(K_{n_1, \ldots, n_r})$ is a principal ideal if and only if $t=n_1+\cdots+n_r$, and hence in this case it can be viewed as the squarefree Veronese ideal of degree $t=n$ in $S$. Therefore, for a given $t$, to be able to characterize $n_1, \ldots, n_r$ such that $I_t(K_{n_1, \ldots, n_r})$ is Cohen-Macaulay,  one only needs to check that when $I_t(K_{n_1, \ldots, n_r})$ is squarefree Veronese of degree $t$ in $S$. More precisely, one needs to check that for which $n_1,\ldots, n_r$ with $n_1+\cdots+n_r=n$, every subset of $[n]$ of size $t$ can be viewed as a path in  $K_{n_1, \ldots, n_r}$. In the case of edge ideal of $K_{n_1, \ldots, n_r}$, the answer immediately follows from \cite[Theorem 4.2]{HH}.
\begin{Proposition}
Let $G$ be a graph such that $I(G)$ is matroidal. Then  $I(G)$  is Cohen-Macaulay if and only if $G$ is a complete graph.
 \end{Proposition}
\begin{proof}
Let $|V(G)|=n$. From  \cite[Theorem 4.2]{HH}, it follows that $I(G)$ is Cohen-Macaulay if and only if $I(G)$ is squarefree Veronese of degree 2 in $S=K[x_1, \ldots, x_n]$, that is, for every $1\leq i\neq j\leq n $, we have $x_ix_j \in I(G)$. Moreover, $x_ix_j \in I(G)$ for every $1\leq i\neq j\leq n $ if and only if $G$ is a complete graph. 
\end{proof}

By using Theorem~\ref{Th.2}, the above proposition can be rephrased as follows:
\begin{Corollary}
The edge ideal $I(K_{n_1, \ldots, n_r})$  is Cohen-Macaulay if and only if $n_1=\cdots=n_r=1$.
 \end{Corollary}
 
 Now we state the main result of this section in which we discuss the general case when $t \geq 2$. 

 \begin{Theorem}\label{CM}
The $t$-path ideal $I_t(K_{n_1, \ldots, n_r}) \neq 0$ is Cohen-Macaulay if and only if $n_i \leq \lceil t/2\rceil$, for each $i=1, \ldots, r$. 
\end{Theorem}
\begin{proof}
Let $V_1, \ldots, V_r$ be the $r$-partition of $G=K_{n_1, \ldots, n_r}$ with $|V_i|=n_i$, for each $i=1, \ldots, r$. Let $n=n_1+\cdots+n_r$. From \cite[Theorem 4.2]{HH} and Remark~\ref{Rem.1} (v), it follows that if $I_t(G)$ is Cohen-Macaulay, then it is a squarefree Veronese ideal of degree $t$ in $S=K[x_1, \ldots, x_n]$, and hence $x_{i_1}\cdots x_{i_t} \in I_t(G)$, for every $1 \leq i_1<\cdots<i_t\leq n$. Moreover, from Remark~\ref{Rem.1}(i), it follows that if there exists some $V_i$ with  $n_i > \lceil t/2 \rceil$, then for any subset $T$ of $V_i$ with $|T|=\lceil t/2 \rceil+1$, there does not exist any $t$-path in $G$ that contains all the elements of $T$. Hence, if there exists some $V_i$ of $G$ with $n_i> \lceil t/2\rceil$, then $I(G)$ is not squarefree Veronese, and hence not Cohen-Macaulay. 
 
 Now assume that $n_i \leq \lceil t/2\rceil$, for each $i=1, \ldots, r$. Let $u=x_{i_1}\cdots x_{i_t}$ be a squarefree monomial in $S$. We need to show that the elements in $A:=\supp(u)$ can be interpreted as a $t$-path in $G$.  Let $A_i= A\cap V_i $ and $a_i = |A_i|$. Since $n_i \leq \lceil t/2\rceil$, we get $a_i \leq \lceil t/2 \rceil$. Without loss of generality, we may assume that the elements in $\{x_{i_1}, \ldots, x_{i_t}\}$ are arranged such that the first $a_1$ elements lie in $V_1$, the next $a_2$ elements lie in $V_2$ and so on. If $t$ is odd, then
 \[
 x_1,x_{\frac{t+1}{2}+1}, x_2, x_{\frac{t+1}{2}+2}, \ldots, x_{\frac{t-1}{2}}, x_t, x_{\frac{t+1}{2}}
 \]
 is a $t$-path in $G$, since consecutive vertices in above sequence belong to different $V_i$'s. Similarly, if $t$ is even, then 
\[
 x_1,x_{\frac{t}{2}+1}, x_2, x_{\frac{t}{2}+2}, \ldots, x_{\frac{t}{2}}, x_t
 \]
 is a $t$-path in $G$. This completes the proof.   
 \end{proof}

\section{limit depth of  $I_t(K_{n_1, \ldots, n_r})$}\label{limdepth}

Let $(R,\mm)$ be a Noetherian local ring or a standard graded algebra over a field $K$ with graded maximal ideal $\mm$. Let $I \subset R$ be an ideal, which is graded if $R$ is standard graded $K$-algebra. It is known from Brodmann \cite{B1} that $\depth (R/I^k)$ stabilizes for large $k$, that is, $\depth(R/I^k)$ is constant for $k \gg 0$. The smallest $t>0$, for which $\depth(R/I^t)= \depth(R/I^k)$ for all $k \geq t$ is called the {\em index of depth stability} of $I$ and is denoted by $\dstab(I)$, as defined in \cite{HRV}. Moreover, $\depth(I^{\dstab(I)})$ is called the {\em limit depth} of $I$ and is denoted by $\lim_{k\rightarrow \infty} \depth(R/I^k)$, see \cite{HH2}. 

 In this section we will study the $\dstab(S/I_t(K_{n_1, \ldots, n_r}))$ and  $\lim_{k\rightarrow \infty} \depth(S/I_t(K_{n_1, \ldots, n_r})^k)$. 
In \cite{HQ}, the linear relation graphs of monomial ideals were introduced to help in understanding their analytic spreads. Let $I\subset S$ be a monomial ideal with $\mathcal{G}(I)=\{u_1, \ldots, u_m\}$. The {\em linear relation graph} $\Gamma$ of $I$ is the graph with the edge set 
\[
E(\Gamma)=\{\{x_i,x_j\}: \text{ there exists } u_k, u_l \in \mathcal{G}(I) \text{ such that } x_iu_k=x_ju_l  \},
\]
and $V(\Gamma)= \bigcup_{\{x_i,x_j\} \in E(\Gamma)} \{x_i,x_j\}$.  It is known from  \cite[Lemma 4.2 and Theorem 4.1]{HQ} that if the linear relation graph of a matroidal ideal $I$ has $m$ vertices and $s$ connected components, then 
\begin{equation}\label{eqdstab}
\dstab(I)<\ell(I)=m-s+1
\end{equation}
where $\ell(I)$ denotes the analytic spread of $I$, that is, the Krull dimension of the fiber ring $\mathcal{R}(I)/\mm\mathcal{R}(I)$. Combining this with \cite[Corollary 3.5]{HRV}, one can deduce that 
\begin{equation}\label{eqlim}
\lim_{k\rightarrow \infty} \depth(S/I^k)= n-(m-s+1).
\end{equation}

 Therefore, to compute the limit depth of $t$-path ideals of complete $r$-partite graphs, it is enough to compute the number of vertices and the number of connected components in their linear relation graphs.
 
  \begin{Remark} \label{remdepthfunction}
   In \cite[Proposition 2.1]{HH2}, it is shown that for any graded ideal $I$, the $\depth (S/I^k)$ is a non-increasing function  of $k$ if all powers of $I$ have a linear resolution. This is indeed the case for matroidal ideals.  Therefore, one concludes the following: if limit depth of a matroidal ideal is $s$ and the $k$ is the smallest integer for which $\depth(S/I^k)=s$, then $\dstab (I)=k$. 
   \end{Remark}
 
  \begin{Remark} \label{remdepth}
  It follows from \cite[Theorem 2.5]{CH} that if $I$ is a fully supported matroidal ideal generated in degree $d$, then $\depth(I)=d-1$. It is shown in Remark~\ref{Rem.1} (iv) that if $I_t(K_{n_1, \ldots, n_r}) \neq 0$, then it is a fully supported ideal. This shows that whenever $I_t(K_{n_1, \ldots, n_r}) \neq 0$, we have $\depth ( I_t(K_{n_1, \ldots, n_r})) =t-1$.
  \end{Remark}
At first, we will compute limit depth of $t$-path ideals of complete bipartite graphs.  

\begin{Theorem}\label{bipartite}
Let $G=K_{p,q}$ with $n=p+q$ and $t \geq 2$. Then we have the following:
\begin{enumerate}
\item[\em{(i)}] if \; $p=\lfloor t/2 \rfloor$ and $q = \lceil t/2 \rceil$, then $\lim_{k\rightarrow \infty} \depth(S/I_t(G)^k)=n-1$ and $\dstab(I_t(G))=1$.

\item[\em{(ii)}]  if $p=\lfloor t/2 \rfloor$ and $q > \lceil t/2 \rceil$, then $\lim_{k\rightarrow \infty} \depth(S/I_t(G)^k)=\lfloor t/2 \rfloor$ and \\ $\dstab(I_t(G))=\big\lceil (q-1)/(q-\lceil t/2 \rceil) \big\rceil$.

\item [\em{(iii)}] if $p>\lfloor t/2 \rfloor$ and $q > \lceil t/2 \rceil$, then $\lim_{k\rightarrow \infty} \depth(S/I_t(G)^k)=0$  and \\ $1< \dstab (I_t(G)) <n $ if $t$ is odd, and $\lim_{k\rightarrow \infty} \depth(S/I_t(G)^k)=1$ and $1\leq \dstab (I_t(G)) <n-1 $, if $t$ is even.

\end{enumerate}
\end{Theorem}
\begin{proof}
Let $V_1=\{x_1, \ldots, x_p\}$ and $V_2=\{y_1, \ldots, y_q\}$ be the $2$-partition of $G$. 

(i) In this case $n=p+q=t$ and $I_t(G)$ is a principal ideal and the assertion follows trivially. 

(ii) Is $p=\lfloor t/2 \rfloor$ and $q>\lceil t/2 \rceil$, then it is evident from Remark~\ref{Rem.1}(i)  that
 all vertices of $V_1$ appear in every $t$-path in $G$. Then, it follows from the definition of $\Gamma$ that $V_1 \cap V(\Gamma)=\emptyset$. It follows immediately from Remark~\ref{Rem.1}(ii) that  $\{y_i, y_j\} \in E(\Gamma)$ for all $1 \leq i \neq j \leq q$.  
Therefore $V(\Gamma) = V_2$ and $\Gamma$ is a complete graph on $q$ vertices, and has only one connected component. We conclude $\ell (I) = q$ and use the equality in (\ref{eqlim}) to compute $\lim_{k\rightarrow \infty} \depth(S/I_t(G)^k)=n-q=p=\lfloor t/2 \rfloor$.

To prove the assertion about $\dstab(I_t(G))$, we first observe that $I_t(G)=JL$, where $J=(x_1\cdots x_p)$ and $L$ is generated by all monomials of degree $\lceil t/2\rceil$ in $q$ variables. Indeed, as discussed before,  all vertices of $V_1$ appear in every $t$-path in $G$, equivalently, $x_1\cdots x_p$ divides every generator of $I_t(G)$. Moreover, for every subset $\{y_{i_1}, \ldots, y_{i_m}\}$ of $V_2$ of size $m=\lceil t/2 \rceil$, there is a $t$-path $P$ in $G$ with $V(P)=\{y_{i_1}, \ldots, y_{i_m}\} \cup V_1$, equivalently, $y_{i_1}\cdots y_{i_m} \in L$. Hence, $J$ is a principal ideal and $L$ is the squarefree Veronese ideal of degree  $\lceil t/2 \rceil$ in variables $y_1, \ldots, y_q$. Let $S_1=K[x_1, \ldots, x_p]$ and $S_2=K[y_1, \ldots, y_q]$. Then $\depth(S_1/J^k)=p-1$ and it follows from  \cite[Corollary 5.7]{HRV} that $\depth(S_2/L^k)=\max\{0,k(\lceil t/2 \rceil-q)+q-1\}$ for all $k$ and  $\lim_{k\rightarrow \infty} \depth(S_2/L^k)=0$ with $\dstab(L)=\lceil (q-1)/(q-\lceil t/2 \rceil) \rceil$.

The fact that the ideals $J$ and $L$ are generated in different sets of variables facilitates to conclude that  $(I_t(G))^k=J^kL^k$, and $\depth (S/I_t(G)^k)= \depth (S_1/J^k) + \depth(S_2/L^k) +1$, for all $k$, for example, see \cite[Lemma 2.2]{HT}. Therefore, 

\[
\depth (S/I_t(G)^k)= p+\max\{0,k(\lceil t/2 \rceil-q)+q-1\}.
\]

Since $\lim_{k\rightarrow \infty} \depth(S/I_t(G)^k)=p$ and $\dstab(L)=\lceil (q-1)/(q-\lceil t/2 \rceil) \rceil$, it follows that $\dstab (I_t(G))=\lceil (q-1)/(q-\lceil t/2 \rceil) \rceil$.

(iii) Let $p > \lceil t/2 \rceil$ and $q > \lfloor t/2 \rfloor$. Then it follows from Remark~\ref{Rem.1}(ii), that $\{x_i, x_j\} \in E(\Gamma)$, for all $1 \leq i\neq j\leq p$ and $ \{y_i, y_j\} \in E(\Gamma)$, for all $1 \leq i \neq j\leq q$.  If $t$ is  even, then every path in $G$ contains half variables from $V_1$ and the other half from $V_2$. Consequently, given any $t$-path in $G$, we cannot replace any vertex from $V_1$ with any vertex from $V_2$. This shows that $\{x_i, y_j\} \notin E(\Gamma)$. Therefore, $\Gamma$ has exactly two connected components, and $\ell(I)=p+q-2+1=n-1$. By using equality in (\ref{eqlim}), we obtain  $\lim_{k\rightarrow \infty} \depth(S/I_t(G)^k)=n-(n-1)=1$ and from (\ref{eqdstab}), we obtain $\dstab(I_t(G)) < n-1$. Remark~\ref{remdepth} gives that $\depth(I_t(G))=t-1$. Hence $\dstab(I_t(G) \geq 1$ and the equality holds only if $t=2$ due to Remark~\ref{remdepthfunction}. 

If $t$ is odd, then set $m=\lfloor t/2 \rfloor$. The sequence $x_1,y_1,x_2,y_2, \ldots ,x_m,y_m$ is a $(t-1)$-path in $G$. It can be extended to a $t$-path by either joining the edge $\{y_{m+1}, x_1\}$, or the edge $\{y_m, x_{m+1}\}$. Hence, for $u=y_{m+1}x_1y_1x_2y_2 \ldots x_my_m$ and $v=x_1y_1x_2y_2 \ldots x_my_mx_{m+1}$, we have $x_{m+1}u=y_{m+1}v$, which gives $\{x_{m+1}, y_{m+1}\} \in E(\Gamma)$. This shows that $\Gamma$ is connected. In fact, by repeating the same argument as above, one can show that $\Gamma$ is a complete graph on $n$ vertices. Therefore, $\ell(I)=n-1+1=n$ and by using the equality in (\ref{eqlim}), we obtain $\lim_{k\rightarrow \infty} \depth(S/I_t(G)^k)=0$ and from (\ref{eqdstab}), we obtain  $\dstab (I_t(G))< n$. Again from Remark~\ref{remdepth} we have $\depth(I_t(G))=t-1>0$. Hence $\dstab(I_t(G) > 1$. This finishes the proof.
\end{proof}

Now we consider the case in which $G$ is an $r$-partite graph with $ r\geq 3$. Let $d, a_1, \ldots, a_n$ be positive integers with $\sum_{i=1}^n a_i\geq d$. Then the ideal generated by all monomials $x_1^{b_1}\cdots x_1^{b_n}$ of degree $d$ with $b_i \leq a_i$, for all $i=1, \ldots, n$ is called the ideal of {\em Veronese type}, and is denoted by $I_{d;a_1, \ldots, a_n}$. In \cite{HRV}, the index of depth stability of Veronese type ideals is discussed in detail. Note that if $a_1=\cdots=a_n=1$, then $I_{d;a_1, \ldots, a_n}$ is squarefree Veronese ideal in $n$ variables. 
From Theorem~\ref{CM} and results obtained in \cite{HRV}, we conclude the following

\begin{Proposition}\label{Sqfrdepth}  
Let $t \geq 2$ and $G=K_{n_1, \ldots, n_r}$ with $|V(G)|=n$ and $n_i \leq \lceil t/2 \rceil$, for all $1 \leq i \leq r$. Then $\lim_{k\rightarrow \infty} \depth(S/I_t(G)^k)=0$ and $\dstab(I_t(G))=\lceil (n-1)/(n-t) \rceil$.
\end{Proposition}


 In the view of Proposition~\ref{Sqfrdepth}, it is enough to consider those cases in which at least one of the $n_i$'s is greater than $\lceil t/2\rceil$. From Theorem~\ref{bipartite}, it is clear that limit depth of $I_t(G)$ depends on $t$ and how vertices are distributed among the partition sets. We see this behaviour in subsequent results too. Now we will describe limit depth and $\dstab$ for $I_3(K_{n_1, \ldots, n_r})$ when $r \geq 3$. 
 Note that since $r \geq 3$, we must have $3 \leq n=n_1+\cdots+n_r$. First we discuss the cases when $n \in \{3,4\}$. The following lemma is a consequence of Proposition~\ref{Sqfrdepth}. 
 
 \begin{Lemma}
 Let $G=K_{n_1, \ldots, n_r}$ with $r \geq 3$. Then we have the following:
 \begin{enumerate}
 \item[\em{(i)}] If $|V(G)|=3$, then $I_3(G)$ is a principal ideal and $\lim_{k\rightarrow \infty} \depth(S/I_3(G)^k)=2$ and $\dstab(I_3(G))=1$.

 \item[\em{(ii)}] If $|V(G)|=4$, then $G$ is either a complete graph on four vertices or isomorphic to $K_{2,1,1}$. Moreover, $\lim_{k\rightarrow \infty} \depth(S/I_3(G)^k)=0$ and $\dstab(I_3(G))=3$.
 \end{enumerate}
 \end{Lemma}

 In the following text, $\mm$ denotes the unique graded maximal ideal in the polynomial ring $S$ whose variables correspond to the vertices of the graph $G$. 

\begin{Proposition}
 Let $G=K_{n_1, \ldots, n_r}$ with $r \geq 3$. Then $\lim_{k\rightarrow \infty} \depth(I(G)^k)=0$ and $\dstab (I(G))=2$.
\end{Proposition}

\begin{proof}
From Remark~\ref{remdepth}, we know that $\depth(I(G))=1$. Then by following Remark~\ref{remdepthfunction}, it is enough to show that $\depth(I(G)^2)=0$. Given a graph $H$, in \cite[Theorem 2.1]{HH3}, the equivalent condition for $\depth(I(H)^2)=0$ is described as follows: $\depth(I(H)^2)=0$ if and only if $H$ contains a cycle of length 3 and every other vertex of $G$ has a neighbour in this cycle. It can be easily seen that $G$ satisfies this condition. Let   $V_1, \ldots, V_r$ be the $r$-partition of $G$ and since $r \geq 3$, we can choose $x_1 \in V_1$,  $x_2 \in V_2$, and $x_3 \in V_3$. Then the subgraph induced by $x_1, x_2$ and $x_3$ is cycle of length 3 in $G$. Moreover, it follows from the definition of $G$ that any $x_j \in V(G)$ is adjacent to at least one of the vertices in  $\{x_1, x_2, x_3\}$. Hence $\depth(I(G)^2)=0$, as required. 
\end{proof}
 
  \begin{Theorem}\label{v>5}
 Let $G=K_{n_1, \ldots, n_r}$ with $r \geq 3$ and $|V(G)|\geq 5$. Then \\$\lim_{k\rightarrow \infty} \depth(I_3(G)^k)=0$ and $\dstab (I_3(G))=2$.
 \end{Theorem}
 
 \begin{proof}
From Remark~\ref{remdepth}, we know that $\depth(I_3(G))=2$, and then by following Remark~\ref{remdepthfunction}, it is enough to show that $\depth(I_3(G)^2)=0$.  Let $V_1, \ldots, V_r$ be the $r$-partition of $G$. 

First, assume that $|V_1|=\cdots=|V_{r-1}|=1$. If $|V_r|=1$ or $|V_r|=2$, then the assertion follows from  Proposition~\ref{Sqfrdepth}. More precisely, If $|V_r|=1$, then $|V(G)|=r$, and if $|V_r|=2$, then $|V(G)|=r+1$. Since $r \geq 5$, in both cases, we have $\dstab(I_3(G))=\lceil (s-1)/(s-3) \rceil= 2$ for $s=r, r+1$. Finally, if $|V_r| >2$, then take $x_1 \in V_1$, $x_2 \in V_2$ and $x_3, x_4 , x_5 \in V_r$, and set $u:=x_1x_2x_3x_4x_5$. Then for any $x_i \in V_r$, we have $ux_i=(x_ix_2x_3)(x_4x_1x_5) \in I_3(G)^2$ because $x_i,x_2,x_3$ and $x_4,x_1,x_5$ are 3-paths in $G$. Furthermore, for any $x_i \in V(G) \setminus V_r$, we have $ux_i=(x_1x_3x_2)(x_4x_ix_5) \in I_3(G)^2$ because $x_1,x_3,x_2 $ and $x_4,x_i,x_5$ are 3-paths in $G$. Therefore, $I_3(G)^2:u=\mm$, and hence $\depth (I_3(G)^2)=0$, as required.

Now, assume that $|V_i|, |V_j| \geq 2$, for some $1 \leq i \neq j \leq r$. Take $x_1,x_2 \in V_i$, $x_3, x_4 \in V_j$, and $x_5 \in V_k$. Set $u:= x_1x_2x_3x_4x_5$; then $u \notin I_3(G)^2$. Then $I_3(G)^2:u=\mm$. Indeed, for any $x_m \in V_i$, we have $ux_m\in I_3(G)^2$ because $x_1, x_5,x_2$ and $x_3,x_m,x_4$ are 3-paths in $G$. Similarly, for any $x_m \in V_j$, we have $ux_m\in I_3(G)^2$. Finally, if $x_m \in V(G) \setminus( V_i \cup V_j )$, then again $ux_m\in I_3(G)^2$ because $x_1,x_m,x_2$ and $x_3, x_5, x_4$ are 3-paths in $G$. This shows that $\depth (I_3(G)^2)=0$, as claimed.
   \end{proof}
Now we give the final result of this section which shows that in the case of $r \geq 3$ and $t \geq 3$, if the number of vertices in each partition set is big enough, then limit depth of $I_t(G)$ is 0.

\begin{Theorem}\label{t=4}
 Let $G=K_{n_1, \ldots, n_r}$ with $r \geq 3$ and $n_i \geq \lceil t/2 \rceil$ for all $1 \leq i \leq r$. Then  $\lim_{k\rightarrow \infty} \depth(I_t(G)^k)=0$ and $1 <\dstab(I_t(G))<n$ for $t \geq 3$ . 
\end{Theorem}

 \begin{proof}
 From Remark~\ref{remdepth}, we conclude that $\depth(I_t(G))=t-1>0$. Let $|V(G)|=n$. Due to  (\ref{eqdstab}) and (\ref{eqlim}), it is enough to prove that the linear relation graph $\Gamma$ of $I_t(G)$ has one connected component and $|V(\Gamma)|=n$. To prove this, we will show that $\Gamma$ is a complete graph with $V(\Gamma)=V(G)$. Let $V_1, \ldots, V_r$ be the $r$-partition of $G$. Let $x_i,x_j \in V(G)$, for some $i \neq j$. We can choose $A\subset V(G)$ such that $|A|=t$, $x_i \in A$, $x_j \notin A$ and $|A \cap V_k| <  \lceil t/2 \rceil$, for all $k=1, \ldots, r$. Such a choice of $A$ is possible because $n_i \geq \lceil t/2 \rceil$ and $r \geq 3$. Take $B=( A \setminus \{x_i\} )\cup \{x_j\}$. Then $|B|=t$ and $|B \cap V_k| \leq  \lceil t/2 \rceil$, for all $k=1, \ldots, r$. Then by following the similar construction of the path at the end of the proof of Theorem~\ref{CM}, we obtain a $t$-path $P$ in $G$ with vertices in $A$ and a $t$-path $P'$ in $G$ with vertices in $B$. Let $u$ be the monomial  in $\mathcal{G} (I_t(G))$ that corresponds to the $t$-path $P$ and  $v$ be the monomial  in $\mathcal{G} (I_t(G))$ that corresponds to the $t$-path $P'$. Then $x_iv=x_ju$ and $\{x_i,x_j\} \in E(\Gamma)$, as required. 

  \end{proof}
 
 \section*{Acknowledgement}
The authors are grateful to the referee for careful reading of the paper and the valuable comments and suggestions.

 
\end{document}